\documentclass[reqno]{amsart}

\usepackage[latin1]{inputenc}

\usepackage{amssymb}
\usepackage{amsfonts}
\usepackage{paralist}
\usepackage{enumitem}
\usepackage{indentfirst}
\usepackage{amsmath}
\usepackage{amsthm}
\usepackage[all]{xy}
\usepackage{mathrsfs}
\usepackage{standalone}
\usepackage{graphicx,psfrag,subfigure}
\usepackage{pst-all}
\usepackage{pifont}
\usepackage[bookmarks]{hyperref}
\usepackage{marvosym}
\usepackage{pifont}
\usepackage{wasysym}
\usepackage{color,soul}
\usepackage[justification=centering]{caption}
\usepackage{comment}
\excludecomment{confidential}
\usepackage{mdwlist} 

\def\R{\mathbb{R}}

\def\T{\mathbb{T}}

\newtheorem*{thm B}{Theorem B}

\newtheorem{theorem}{Theorem}
\newtheorem{question}{Question}

\newtheorem{cor-thm}{Corollary}[theorem]

\newtheorem{prop}{Proposition}[subsection]
\newtheorem{lemma}{Lemma}[subsection]
\newtheorem{cor}{Corollary}[section]
\newtheorem{rmk}{Remark}[subsection]

\DeclareMathOperator{\diag}{\mathrm{diag}}

\begin{document}
\title[New classes of $C^1$-robustly transitive]
{New classes of $C^1$-robustly transitive maps with persistent critical points}

%

 \email{clizana@ufba.br} 
 \email{wagnerranter@im.ufal.br}


\maketitle

\centerline{\scshape Cristina Lizana}
\medskip
{\footnotesize
 \centerline{ Departamento de Matem\'{a}tica. Instituto de Matem\'{a}tica
	e Estat\'{i}stica.}
	 \centerline{Universidade Federal da Bahia. Av. Adhemar de Barros s/n, 40170-110. 
	Salvador/Bahia, Brazil.}
} 

\medskip

\centerline{\scshape Wagner Ranter}
\medskip
{\footnotesize
 \centerline{Departamento de Matem\'{a}tica. Universidade Federal de Alagoas.}
   \centerline{Campus A.S. Simoes s/n, 57072-090. Macei\'o/Alagoas, Brazil.}
  \centerline{Postdoctoral fellow at Math Section of ICTP. Strada Costiera, 11. I-34151, Trieste, Italy.}
}

\bigskip

\begin{abstract}
We exhibit a new large class of $C^1$-open examples of robustly transitive maps displaying persistent critical points in the homotopy class of expanding endomorphisms acting on the two dimensional torus and the Klein bottle.	
\end{abstract}

\section*{Introduction}\label{intro}
\let\thefootnote\relax\footnotetext{AMS classification:  37D30, 37D20, 08A35, 35B38. Key words: robustly transitive endomorphisms, critical set.}

A map is \textit{$C^1$-robustly transitive} if there is a dense forward orbit, that is \textit{transitivity},
for every map in a $C^1$-neighborhood. This issue has been focus of attention by several authors.

For diffeomorphisms, it is well-known that the existence of $C^1$-robustly transitive maps requires some weak hyperbolicity structure. For further details see \cite{Mane-Closinglemma, DPU}, and \cite{BDP}. It is  widely known that in dimension two this structure imply topological obstructions. For instance, the only surface supporting a $C^1$-robustly transitive diffeomorphism is the torus, and it must be homotopic to a hyperbolic linear diffeomorphism.

For  non-invertible maps (\textit{endomorphisms}), on one hand \cite{LP} showed that any weak form of hyperbolicity is needed for the existence of $C^1$-robustly transitive local diffeomorphisms (covering maps). On the other hand, we exhibit in \cite{CW1} some topological obstructions for the  existence of $C^1$-robustly transitive surface endomorphisms displaying \textit{critical points}, that is points such that the derivative is not surjective. More precisely, it was proved in \cite{CW1} that if $F$ is a robustly transitive surface endomorphism displaying critical points, then $F$ has some weak form of hiperbolicity, namely \textit{partially hyperbolic}, which, roughly speaking, means that there exists a family of cone field on the tangent bundle which is preserved and its vectors are expanded by the derivative.
Consequently, one has that the only surfaces supporting $C^1$-robustly transitive endomorphisms are the Torus and the Klein bottle. Furthermore, it is shown that the action of every $C^1$-robustly transitive endomorphism on the first homology group has at least one eigenvalue of modulus greater than one. In other words, every $C^1$-robustly transitive surface endomorphism is homotopic to a linear endomorphisms having  at least one eigenvalue of modulus greater than one. In particular, there are  not $C^1$-robustly transitive endomorphisms homotopic to the  identity. Hence, some natural questions arise. 

\begin{question}\label{q1}
	What are the homotopy classes admitting robustly transitive endomorphisms?
\end{question}

Since some examples have been constructed recently in the 2-torus, another question is the following.

\begin{question}\label{q2}
	Does the Klein bottle admit robustly transitive maps?
\end{question}

Before doing some comments about the state of the art related to these questions, let us fix some notations. Throughout this paper $\mathrm{End}^1(M)$ denotes the set of all endomorphisms over $M,$ where $M$ is either the torus $\mathbb{T}^2$ or the Klein bottle $\mathbb{K}^2$, and the set of all critical points of an endomorphism $F$ will be denoted by $\mathrm{Cr}(F)$. Consider the linear endomorphism induced by the matrix $L$ with all integers entries, which, by slight abuse of notation, we denote also by $L$, and $\mu, \lambda \in \mathbb{R}$ its eigenvalues. 

For endomorphisms in $\mathrm{End}^1(\mathbb{T}^2)$ without critical points, besides the classical examples as linear hyperbolic endomorphisms and expanding endomorphisms which have some hyperbolic structures, some examples appear in \cite{Sumi, He-Gan} and \cite{LP}. However, these examples cannot be extended to obtain robustly transitive endomorphisms displaying critical points. The problem is, up to a perturbed, under the existence of critical points, some open sets are sent onto a curve; and the classical argument for proving the robust transitivity via open set, that is given two open set there is an iterate of one of them intersecting the other one, could failed in the $C^1$-topology under the existence of critical points.

The first example of robustly transitive endomorphisms displaying critical points was given in \cite{Berger-Rovella}, this example is homotopic to a hyperbolic linear endomorphism, that is, it is in the homotopy class of $L$ with $|\mu|<1<|\lambda|$. Later,  \cite{ILP} exhibited new classes of examples which are in the homotopy class of a linear expanding endomorphism $L$.
Unfortunately, there is a mistake in the proof for showing the robust transitivity of the constructed map in the homology class of $L$ with $1<|\mu|<|\lambda|$. Let us be more precise. In the proof of Proposition 2.1 in \cite{ILP}, they start with a linear endomorphism $L$ with $1<|\mu|<|\lambda|$ and obtain  $F \in \mathrm{End}^1(\mathbb{T}^2)$ a perturbation $C^0$-close to $L$, but $C^1$-far from $L$, having persistent critical points and preserving the unstable cone field naturally defined for $L$. However, in order to prove the robust transitivity of $F$, they use the fact that all images of any open set by $F$ has nonempty interior, which is not always satisfied for endomorphisms displaying critical points as we commented above. 

One of the goal of this paper is to be a corrigendum of \cite{ILP} and come to provide that such examples can in fact be constructed. Our approach is slightly bit different from the one used in \cite{ILP}. Let us state the first main result, for this consider $L$ a linear endomorphism defined on the torus.

\begin{theorem}\label{thm A}
Let $L$ be a linear endomorphism whose eigenvalues are $\mu, \lambda \in \mathbb{R}$ so that $1<|\mu|\ll|\lambda|$. Then there are $C^1$-robustly transitive endomorphisms homotopic to $L$ displaying critical points that are persistent under small perturbations.
\end{theorem}

The condition $1<|\mu|\ll |\lambda|$ means that $|\lambda|$ is larger enough than $|\mu|$ 
such that the correspondent eigenspaces generate a dominated splitting
. This will be useful to create a region with a ``good mixing'' property which is persistent under small perturbations, so-called ``blender''.

The proof is divided in two parts. The first part corresponds to the case the eigenvalues $\lambda$ and $\mu$ are integer numbers, which we call Periodic case, see Section~\ref{periodic}. The second one corresponds to the case  the eigenvalues are irrational, so-called Irrational case, see Section~\ref{IC}.

Roughly speaking the idea of the construction is as follows. Start with an expanding linear endomorphism $L$ with $1<|\mu|\ll|\lambda|$. Make a deformation in the weaker direction in order to get a ``blender''.  The new homotopic map is robustly transitive endomorphism.
Then, create artificially critical points far away from the ``blender'', which are persistent under small $C^1$-perturbations, in such a way that the transitivity.

\smallskip

The blenders were first introduced in \cite{BD} as a mechanism to create robustly transitive non-hyperbolic diffeomorphisms. Blenders were used in \cite{He-Gan} and \cite{LP} to create robustly transitive non-hyperbolic local diffeomorphisms on the 2-torus. The definition of blender given in \cite{He-Gan} is quite technical. There are several authors using the knowledge of blender in the non-invertible setting, for instance see \cite{Berger}.  We are assuming the notion of blender as follows, which  is an adaptation of the definition introduced in \cite{He-Gan}, doing the consideration that  for us the stable and unstable directions are on the contrary as it is in \cite{He-Gan}, that is, for us the stable direction is horizontally and the unstable direction is  vertical.

\smallskip

Let $I, J$ two closed intervals in $S^1.$ Consider a rectangle $\mathcal{Q}=I\times J$ in $\mathbb{T}^2$ and $F:\mathcal{Q}\to \mathbb{T}^2$
a $C^1$ map. 
$(\mathcal{Q}, F)$ is a {\it blender} if there are two compact subsets $R_1$ and $R_2$ of $\mathcal{Q}$ such that
$F|_{R_i}$ is a diffeomorphism onto its image, for $i=1,\,2$, and $\mathcal{Q}\subseteq F(R_1)\cup F(R_2)$  verifying:

\begin{enumerate}[label=(\alph*)]
	\item[(B1)]  $F$ is hyperbolic on $R_i$, with $i=1,\, 2$. That is, 
	denoting by $E^s$ and $E^u$ the tangent bundle to $I$ and $J$, respectively, for all $(x,y) \in \mathcal{Q}$  one has that  the unstable cone field $$\mathcal{C}^u_{\alpha}(x,y)=\{u+v \in E^s(x,y) \oplus E^u(x,y): \|u\| \leq \alpha \|v\|\},$$
	for $\alpha >0$ small, satisfies:
	
	\smallskip
	
	\begin{enumerate}
		\item[($i$)]  $\overline{DF(\mathcal{C}^u_{\alpha}(x,y))} \subseteq \mathrm{int}(\mathcal{C}^u_{\alpha}(F(x,y)))$; 
		\item[($ii$)] $T_{(x,y)}\mathbb{T}^2= E^s(x,y) \oplus \mathcal{C}^u_{\alpha}(x,y)$; 
		\item[($iii$)] there exists $1< \sigma < \lambda$ such that
		\begin{align}
		\|DF(u)\|<\sigma^{-1}\|u\| \ \ \text{and} \ \ \|DF(w)\|\geq \sigma \|w\|,
		\end{align}
		for all $u \in E^s(x,y)$ and $w \in \mathcal{C}^u_{\alpha}(x,y)$;
	\end{enumerate}
	\item[(B2)] there exist two \textit{u-arc} $\gamma_1$ and $\gamma_2$ (that is, $\gamma_i'$ is contained in $\mathcal{C}^u_{\alpha}$) of $\mathcal{Q}$ such that  $F(\gamma_1\cap R_i)\supseteq \gamma_i$ and $\gamma_2$ is strictly on one side of $\gamma_1$;
	\item[(B3)] let $V_1$ be the closed subset of $\mathcal{Q}$ between $\gamma_1$ and $\gamma_2$, $V_2$ a closed subset of $\mathcal{Q}\setminus V_1$ such that $\gamma_2$ is part of the boundary of $V_2$ and
	$R_i'=(R_i\cap V_1)\cup (R_i\cap V_2)$ subset of $\mathcal{Q},$ for $i=1,\,2,$ then $F(R_i')$ contains $V_i,\, i=1,\, 2$.
\end{enumerate}

\smallskip

Note that since $F(\gamma_1 \cap R_1) \supseteq \gamma_1$, there exists a hyperbolic saddle fixed point $p_F\in \gamma_1$.  
In \cite[Proposition 3.2]{He-Gan} is proved that there exists a neighborhood $\mathcal{U}_F$ of $F$ so that $(\mathcal{Q}, G)$ is a blender for every $G \in \mathcal{U}_F$. Furthermore, the  closure of the unstable manifold $\mathcal{W}^u(p_G,G)=\cup_{n\geq 0}G^n(\mathcal{W}^u_{loc}(p_G,G))$ has nonempty interior, where $p_G$ is the continuation of $p_F$.

From the proof of Theorem \ref{thm A}, we are able to construct on the Klein bottle a $C^1$-robustly transitive endomorphism having critical points that are persistent under small perturbations, answering Question~\ref{q2} above affirmatively getting the second goal of this work.

Before state next result, let us make some comments  about the construction of the examples on the Klein bottle. Let $\alpha, \beta : \mathbb{R}^2 \to \mathbb{R}^2$ defined by $\alpha(x,y)=(x+1,y), \beta(x,y)=(-x,y+1),$ and $\Gamma$ be the group of self-homeomorphisms of $\mathbb{R}^2$ generated by $\alpha$ and $\beta$. The {\it{Klein bottle}} is defined as the quotient space given by $\mathbb{K}^2=\Gamma \backslash \mathbb{R}^2$.
Consider the diagonal matrix $L=\diag(\mu,\lambda)$, where $\mu, \lambda$ are nonzero integers and $\lambda$ is odd. This matrix induces a linear endomorphism on the Klein bottle, so-called $L$, given by $L[x,y]=[\mu x, \lambda y], \forall [x,y] \in \mathbb{K}^2$. For further details see \cite{JKM}.

We now are able to state the result.


\begin{theorem}\label{thm B}
Suppose that $1< |\mu|\ll |\lambda|$. Then, there is $C^1$-robustly transitive endomorphism on $\mathbb{K}^2$ displaying (persistent) critical points homotopic to a linear endomorphisms $L$. 
\end{theorem}

The proof is omitted since it is basically the same as for the torus case, 
see  Section~\ref{periodic}.
We emphasize that our approach requires a strong dominated condition 
 which is provided by the assumption $|\mu|\ll|\lambda|$. 
Some  cases remain open as follows. 

\begin{question}
Are there examples of $C^1$-robustly transitive endomorphism displaying critical points in the homotopy class
of a homothety?
\end{question}

\begin{question}
	Can a zero degree endomorphism be $C^1$-robustly transitive?
\end{question}

We believe that both questions can be answer affirmatively. For further discussions about this issue we suggest the readers to see \cite{CW1}. The paper is organized as follows. Section~\ref{preliminaries} is devoted to the basic notions and the prototype for the example. Section~\ref{construction} is dedicated to the construction of the new large class of example satisfying Theorem \ref{thm A} for both cases.


\section{Preliminaries}\label{preliminaries}

Throughout this section, we introduce some essential notation and terminologies that will be used for proving  Theorem~\ref{thm A}.

\subsection{Iterated Function Systems - IFS} 
Given $f_1, \dots, f_{l}:I \to I$ orientation preserving  maps on an interval $I$, not necessarily invertible, we defined as IFS of $f_1,\dots,f_l$ the set of all possible finite compositions of $f_i$'s, that is, $$<f_1,\dots, f_l>:=\{h=f_{i_m}\circ\dots\circ f_{i_1}:\; i_{k}\in\{1,\dots,l\},\, 1\leq k \leq m,\, \mbox{and}\ \ m\geq 0\}.$$ The \textit{orbit} of $x$ is given by $\mathcal{O}(x):=\{h(x):\, h \in\, <f_1,\dots, f_l> \}$, and the length $|h|$ of the word $h$ by $m$ if $ h=f_{i_m}\circ\dots\circ f_{i_1}$.
\subsection{Prototypes}\label{sec:prototype}
By prototype, we mean a particular local feature of endomorphism which exhibits a blender. In our approach, we will deform the linear endomorphism in order to obtain locally one of the following prototypes.

Suppose that $F: \mathbb{T}^2 \to \mathbb{T}^2$ is an endomorphism admitting a box $\mathcal{Q}=I\times J \subset \mathbb{T}^2$, and two disjoint intervals $J_1$ and $J_2$ contained in $J$ so that 
$$F(x,y)=(f_i(x), \lambda y \,\, ({\rm{mod}}\,\, 1)), \forall (x,y) \in I\times J_i,$$
satisfying the following conditions:
\begin{enumerate}[label=(P.\arabic*)]
\item $\lambda >1$ such that $\lambda J_i:=\{\lambda y:y \in J_i\}\supset J$;
\item $f_1, \,f_2:I\to I$ are as one of the cases in Figures \ref{IFS-blender-1} and \ref{IFS-blender-2}.
\end{enumerate}

\begin{figure}[h!]
\begin{minipage}{0.45\linewidth}
\centering
\includegraphics[scale=0.6]{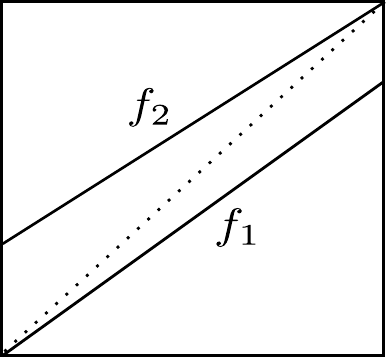}
\caption{Periodic case.}\label{IFS-blender-1}
\end{minipage}
\begin{minipage}{0.45\linewidth}
\centering
\includegraphics[scale=0.6]{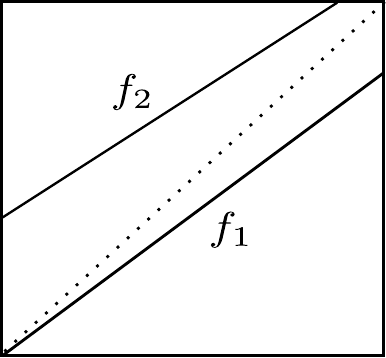}
\caption{Irrational case.}\label{IFS-blender-2}
\end{minipage}
\end{figure}

We claim that $(\mathcal{Q}, F)$ is a blender according to the definition given in the Introduction. 
In fact, for both cases (Figures \ref{IFS-blender-1} and \ref{IFS-blender-2}), we have that $F$ is hyperbolic on $R_i=I\times J_i$, with $i=1,\, 2$, since there exists $\alpha>0$ such that for all $(x,y) \in R_i, \, i=1,\,2$:
\begin{enumerate}[label=(\roman*)]
	\item $T_{(x,y)}\mathbb{T}^2=E^s(x,y)\oplus \mathcal{C}^u_{\alpha}(x,y)$;
	\item $DF(x,y)$ is the diagonal matrix $\diag(f_i(x),\lambda)$ on $E^s(x,y)\oplus E^u(x,y)$;
	\item there exists $1< \sigma < \lambda$ such that:
	\begin{align}
	|f_i'(x)|<\sigma^{-1} \ \ \text{and} \ \ \|DF_{(x,y)}(w)\|\geq \sigma, \, \forall w \in \mathcal{C}^u_{\alpha}(x,y).
	\end{align} 
\end{enumerate}
Then, item (B1) holds.
Moreover, the $u$-arcs $\gamma_1$ and $\gamma_2$ defined as $\{x_0\}\times J$ and $F(\{x_0\}\times J_2)$, respectively, with $x_0$ a fixed point for $f_1,$ satisfy item (B2). Item (B3) follows easily from the construction. That concludes the claim.

Recall that if $(\mathcal{Q}, F)$ is a blender then there is a saddle point $p_F$ such that the closure of the unstable manifold $\mathcal{W}^u(p_F)$ of $F$ has nonempty interior. Furthermore, there is a neighborhood $\mathcal{U}_F$ of $F$ in $\mathrm{End}^1(\mathbb{T}^2)$ so that for every $G \in \mathcal{U}_F$, we can see that the unstable manifold $\mathcal{W}^u(p_G)$ of $G$ at the continuation point $p_G$ of $p_F$ also has nonempty interior.

\begin{rmk}\label{s-mfld}
	Observe that the local stable manifold $\mathcal{W}^s_{loc}(p_F)$ for $F$ at the hyperbolic fixed saddle point $p_F$ contains the interval $I$. Recall that the local stable manifold $\mathcal{W}^s_{loc}(p_G,G)$ depends continuously on $G$ in $\mathcal{U}_F$. 
\end{rmk}


\section{Construction of the examples}\label{construction}
In this section, we are dedicated to prove Theorem \ref{thm A}. The proof is split into two parts. The first part, it is the periodic case, when the eigenvalues $\mu$ and $\lambda$ are integers. The second one, it is when $\mu$ and $\lambda$ are irrational. Recall that $\mu$ and $\lambda$ are eigenvalues of a square matrix $L$ with all integer entries.

Let us fix some notation and terminologies. Write $E^c$ and $E^u$ to denote the subspaces associated to the eigenvalues $\mu$ and $\lambda$, respectively. We denote $\mathcal{F}^c$ and $\mathcal{F}^u$ the foliations induced on the torus by the eigenspaces $E^c$ and $E^u$, called \textit{(weak) unstable foliation} and \textit{strong unstable foliation} respectively. The leaves of $\mathcal{F}^c$ and $\mathcal{F}^u$ are preserved by the linear endomorphism $L$. 

We define for each $\alpha >0$ and $p \in \mathbb{T}^2$ the \textit{unstable cone at point p} by:
\begin{align}\label{coneL}
\mathcal{C}_{\alpha}^u (p)=\{(u,v) \in E^c\oplus E^u: \  \|u\|\leq\alpha\|v\|\}.
\end{align}
It easy to see that for every $\alpha >0$, we have $L(\mathcal{C}_{\alpha}^u)$ is contained in $\mathcal{C}_{\mu \alpha/\lambda}^u$. 

\subsection{Periodic case}\label{periodic}
We here assume that $\mu$ and $\lambda$ are integers and that the modulus of $\lambda$ is sufficiently larger than the modulus of $\mu$ to provide the dominated structure. It easy to see that, in this case, the unstable and strong unstable foliations consist of closed curves. For the sake of simplicity,  we may assume without loss of generality that the leaves of $\mathcal{F}^c$ and $\mathcal{F}^u$ are \textit{horizontal and vertical closed curves} respectively. That is, the leaves of $\mathcal{F}^c$ and $\mathcal{F}^u$ are the following closed curves:
\begin{align}
\mathcal{F}^c(x,y)=\mathbb{S}^1 \times \{y\} \ \ \text{and} \ \ \mathcal{F}^u(x,y)=\{x\}\times \mathbb{S}^1, \forall (x,y) \in \mathbb{T}^2=\mathbb{R}^2/\mathbb{Z}^2.
\end{align}


\subsubsection{One-dimensional dynamics}\label{1dd}
Take $f_0:\mathbb{S}^1\to \mathbb{S}^1$ as $f_0(x)=\mu x \,\, (\mathrm{mod}\,1)$. For every $\varepsilon >0$ small enough, we can deform $f_0$ to obtain two maps $f_1$ and $f_2$ which restricted to the interval $(-\varepsilon, \varepsilon)\subseteq \mathbb{S}^1$ are contracting affine maps as in Figure~\ref{ISF-degree>1}.
\begin{figure}[!htb]
	\centering
	\includegraphics[scale=0.7]{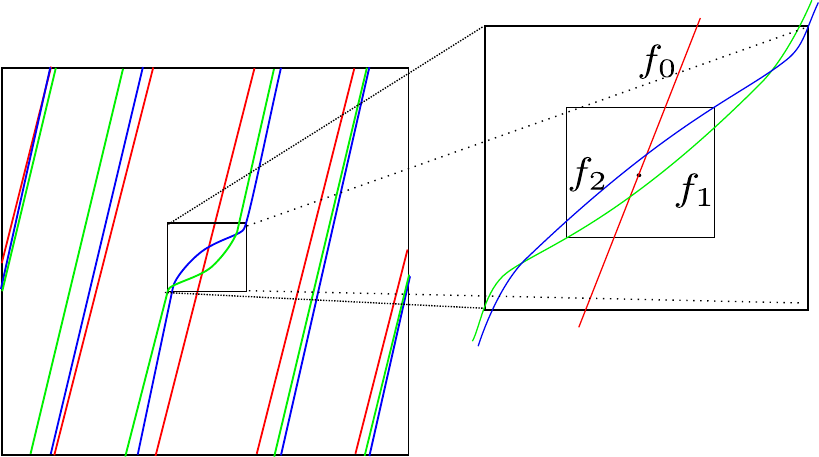} 
	\caption{The graphs of $f_1$ and $f_2$ with topological degree $\mu$.}
	\label{ISF-degree>1}
\end{figure}

\noindent
Furthermore, the maps $f_1$ and $f_2$ are orientation preserving $C^1$-local diffeomorphisms on $\mathbb{S}^1$ of degree $\mu$ and verify the following:

\begin{enumerate}[label=(\arabic*)]
	\item $f_i$ is $2\varepsilon \, C^0$-close to $f_0(x):=\mu x\, ({\rm{mod}}\,1)$ with $0<|\partial_x f_i|< 1$ in $(-\varepsilon,\varepsilon)$;
	\item there exist points in $\mathbb{S}^1$, $-\varepsilon< p_0^1<0< p_0^2<\varepsilon$, where $f_i$'s restricted to $I=[p_0^1,p_0^2]$ are affine maps and $p_0^i$ is an attractor fixed point of $f_i$, $i=1,2$;
	\item  $f_i$ has  two repeller fixed points  $p^i_{-}$ and $p^i_{+}$ in a neighborhood of the attractor $p^i_0$ satisfying $p^i_{-}<-2\varepsilon< p^i_0<2\varepsilon< p^i_{+}$.
\end{enumerate}

Note that $f_1$ and $f_2$ restricted to $I$ are as in Figure \ref{IFS-blender-1}.

\begin{enumerate}[resume]
\item the map $f_i(x)$ restricted to $\mathbb{S}^1\backslash (-2\varepsilon,2\varepsilon)$ is an affine function of the form $\alpha_i x +\beta_i$, where $\alpha_i$ is sufficiently close to $\mu$ and $\beta_i$ is chosen such that its forward orbit by $f_0,\, \mathcal{O}^{+}(\beta_i,f_0)=\{f_0^k(\beta_i):k\geq 0\}$, is dense in $\mathbb{S}^1$.
\end{enumerate}

We can assume $\partial_x f_i<\mu'$, for some $\mu'$ close to $\mu$ but $\mu'> \mu$, and $i=1,2$.

\subsubsection{Two-dimensional dynamics}\label{skewproduct}
We define $F:\mathbb{T}^2 \to \mathbb{T}^2$ as a skew product map $F(x,y)=(f_y(x),\lambda y \, \, (\mathrm{mod}\,\, 1))$, where $f_y:\mathbb{S}^1 \to \mathbb{S}^1$ is an homotopy map satisfying:
\begin{align}
f_y(x)=\left\{\begin{array}{ll}
f_i(x)&,\, (x,y) \in \mathcal{R}_i \ \ \text{for all} \,\, i=0,1,2;\\
f_0(x) &, \, (x,y) \in \mathbb{T}^2 \backslash \mathcal{R}_J,
\end{array}\right.
\end{align}
where $\mathcal{R}_J$ and $\mathcal{R}_i$ are horizontal stripes in $\mathbb{T}^2$ of the form $\mathbb{S}^1 \times J$ and $\mathbb{S}^1 \times J_i$, respectively; the interval $J$ in $\mathbb{S}^1$ centered at the origin (the fixed point of $f_0$), and $J_0,J_1$ and $J_2$ are three pairwise disjoint intervals contained in $J$ sufficiently small such that $\lambda J_i =\{\lambda x : x \in J_i\} \supseteq J$. Then, we should verify that $F$ restricted to the box $\mathcal{Q}=I\times J$, where $I=[p_0^1,p_0^2]$, is as the prototype in Figure \ref{IFS-blender-1} which implies that $(\mathcal{Q},F)$ is a blender.

These intervals $J_i$ can be easily chosen for large $\lambda$ ( e.g., $\lambda\geq 5$). We are interested to create  $C^1$-robustly transitive endomorphisms displaying critical points homotopic to expanding linear endomorphisms, but not interested to quantify exactly which $\lambda$ should be.

\subsubsection{Weak hyperbolicity}
Here we show that the skew product map $F$ defined in the Section \ref{skewproduct} has a hyperbolic structure. In order to prove that, we show that $F$ preserves the unstable cone field $\mathcal{C}_{\alpha}^u$ defined in \eqref{coneL}. 

We need some control on the derivative respective to the second variable for constructing an invariant family of cones. The following result shows that we may choose the homotopy $f_y$ so that $|\partial_y f_y|$ is sufficiently small.

\begin{prop}\label{homotopy}
	Given $\delta>0,$ there exists $\varepsilon>0$ such that for $f, g: \mathbb{S}^1\to \mathbb{S}^1$ $C^1$-maps
	of topological degree $\mu$
	which are  $2\varepsilon$ $C^0$-close, there exists an homotopy $H_t$ between $f$ and $g$ so that $|\partial_t H_t|<\delta.$
\end{prop}

\begin{proof}
	Let $\pi:\mathbb{R} \to \mathbb{S}^1$ be the universal covering map.
	Let $\tilde{f}$ and $\tilde{g}$ be the lift of $f$ and $g$, respectively, such that $\tilde{f}$ and $\tilde{g}$ are $\varepsilon$ $C^0$-close each other.
	Then, we may define the isotopy from $\tilde{f}$ to $\tilde{g}$ by $\tilde{H}_t(\tilde{x})=(1-t)\tilde{f}(\tilde{x})+t \tilde{g}(\tilde{x})$,
	for every $\tilde{x} \in \mathbb{R}, t \in [0,1]$.
	Since $\tilde{H}_t(\tilde{x}+m)
	=	\tilde{H}_t(\tilde{x})+\mu m$,	for every $ \tilde{x} \in \mathbb{R},\,m \in \mathbb{Z},$
	we have that the homotopy $H_t(x)=\pi \circ \tilde{H}_t(\tilde{x})$ between $f$ and $g$ is well defined.
	Furthermore,
	$$|\partial_t H_t|\leq (\max_{\tilde{x}\in \mathbb{R}}{|D\pi(\tilde{x})|}) |\partial_t \tilde{H}_t|.$$ 
	Since 
	$|\partial_t \tilde{H}_t|=|\tilde{f}(\tilde{x})-\tilde{g}(\tilde{x})|<2\varepsilon$ and $\max_{\tilde{x}\in \mathbb{R}}{|D\pi(\tilde{x})|}$
	is bounded, we may choose $\varepsilon > 0$ small enough such that $|\partial_t H_t|<\delta$.
\end{proof}

We now assume that $f_y$ is a homotopy that $|\partial_y f_y|<\delta$. Next, it will be specified some restriction over $\delta>0$ to prove the hyperbolic structure for $F$.

\begin{lemma}[Unstable cone fields for $F$] \label{open-conefields}
Given $\delta>0$ and $\alpha >0$, there exist $0<\theta:=\theta(\delta,\alpha)<1$ and $1<\lambda' <\lambda$ such that the following properties hold:
\begin{enumerate}[label=$\mathrm{(\alph*)}$]
\item$\overline{DF_p(\mathcal{C}_{\alpha}^u(p))}\backslash\{(0,0)\} \subset \mathcal{C}_{\theta \alpha}^u(F(p))$,
for every $p \in \T^{2}$;
\item if $w\in \mathcal{C}_{\alpha}^u(p),$ then $\|DF_p(w)\|\geq \lambda'\|w\|$.
\end{enumerate}
\end{lemma}

\begin{proof}
The proof follows easily from the facts $|\partial_y f_y|<\delta$ and the derivative of $F$ is
$$DF=\left(\begin{array}{cc}
\partial_x f_y & \partial_y f_y \\
0 & \lambda
\end{array}\right)$$
with $|\partial_x f_y| + \delta < \lambda$.
\end{proof}

It is well known that an unstable cone field is shared by every endomorphism $C^1$-close enough to $F$. That is, there exists a neighborhood $\mathcal{U}_F$ of $F$ in $\mathrm{End}^1(\mathbb{T}^2)$ so that the unstable cone field $\mathcal{C}^u_{\alpha}$ so is an unstable cone field for any $G$ in $\mathcal{U}_F$. Thus, it follows direct from Lemma \ref{open-conefields} the following statement.

\begin{cor}
There exists $1<\lambda'<\lambda$ such that for every $G \in \mathcal{U}_F$ and every u-arc $\gamma$ (that is, $\gamma' \subset \mathcal{C}^u_{\alpha}$), holds that $\ell(G^n(\gamma))\geq (\lambda')^n \ell(\gamma)$ for each $n\geq 1$. 
\end{cor}

Next statement guarantees that every arc tangent to the unstable cone field has some iterate ``almost" like the leaves of $\mathcal{F}^u_i$ which is the connected component of the unstable foliation $\mathcal{F}^u$ in $\mathcal{R}_i$.

\begin{prop}\label{lambda-lemma}
For every $\tau >0$ small enough, we can take $\mathcal{U}_F$ small enough such that for every $G$ in $\mathcal{U}_F$ and every $u$-arc $\gamma$ (that is, $\gamma' \subset \mathcal{C}^u_{\alpha}$), there exist a sequence of points $(q_n)_n$ and a sequence of $u$-arc $(\gamma_n)_n$ in $\mathcal{R}_i$ satisfying:
\begin{enumerate}[label=$(\roman*)$]
\item $q_n \in \gamma_n, \, \gamma_0\subset \gamma,$ and $\gamma_n\subset G(\gamma_{n-1})$;
\item $(\gamma_n)_n$ is $\tau \, \, C^1$-close to $\mathcal{F}^u_i(q_n)$ for large $n$.
\end{enumerate}
\end{prop}


\begin{proof}
First, note that the cone $\mathcal{C}^u_{\alpha}$ is an unstable cone field for $G$ and the curve $G^n(\gamma)$ grows up transversely to the weak unstable foliation $\mathcal{F}^c$. Then, up to get an iterate of $\gamma$, we can assume $\gamma$ cross the stripe $\mathcal{R}_i$. Furthermore, we get $\gamma_n$ as the component of $G(\gamma_n)$ in $\mathcal{R}_i$, where $\gamma_0$ is the $u$-arc $\gamma$. Each curve $\gamma_n$ is an $u$-arc and intersects each leaf of $\mathcal{F}^c$ in the stripe $\mathcal{R}_i$ exactly once.

Since $F(x,y)=(f_i(x),\lambda y  \,\, (\mathrm{mod}\,\, 1))$ in $\mathcal{R}_i$, we can shrink $\mathcal{U}_F$, if necessary, to get that for every $G=(g^1,g^2) \in \mathcal{U}_F$ holds that
$|\partial_x g^1|< \mu', |\partial_y g^1|, |\partial_x g^2|<\tau $, and $|\partial_y g^2-\lambda|< \tau$.   

Let $v_0=(v_0^c,v_0^u)$ be an unit vector in $\mathcal{C}_{\alpha}^u(q_0)$ with slope $\rho_0=|v_0^c| / |v_0^u|$. We take $q_n=G^n(q_0) \in \mathcal{R}_i$, for each $n \geq 0$, and $v_n=DG(v_{n-1}) \in \mathcal{C}_{\alpha}^u(q_n)$ with slope $\rho_n$. Then, we can see inductively that	
\begin{align}\label{eq:key}
\rho_{n}=\dfrac{|v_n^c|}{|v_n^u|}=\dfrac{|\partial_x g^1(q_0)v^c_{n-1}+\partial_y g^1(q_0)v^u_{n-1}|}{|\partial_x g^2(q_0)v^c_{n-1}+\partial_y g^2(q_0)v^u_{n-1}|}
\leq \dfrac{\mu'\rho_{n-1}+\tau}{|\lambda-\tau|-\tau\rho_{n-1}}.
\end{align}
Since $\overline{DG(\mathcal{C}_{\alpha}^u(q))}\backslash \{(0,0)\}\subset \mathcal{C}_{\alpha}^u(G(q))$, we have $\rho_n< \alpha$ and
\begin{align*}
\rho_{n} &\leq \dfrac{\mu'\rho_{n-1}+\tau}{|\lambda-\tau|-\tau\rho_{n-1}}\leq \dfrac{\rho_{n-1}}{b}+\dfrac{\tau}{\mu'b}
\end{align*}
where $b=(|\lambda-\eta|-\tau\alpha)/\mu'>1$. Then,
\begin{align*}
\rho_{n} \leq \dfrac{\rho_{0}}{b^n}+\dfrac{\tau}{\mu'}\sum_{i=1}^{n}\dfrac{1}{b^i} = \dfrac{\rho_{0}}{b^n}+\dfrac{\tau(1-b^{-n})}{\mu'(b-1)}.
\end{align*}
By compactness, the vector $v_0$ can be chosen so that $\rho_0$ is the maximum possible slope of unit vectors in $\mathcal{C}^u_{\alpha}$. This together with the fact that $G(\|\gamma_n'(q)\|)\geq \lambda'\|\gamma_n'(q)\|$ guarantee that $\gamma_n$ is $\tau$ $C^1$-close to $\mathcal{F}_i$ for large $n$.
\end{proof}

\subsubsection{Dynamics properties for IFS}
Since $F(x,y)=(f_i(x), \lambda y \, (\mathrm{mod}\,\, 1))$ for all $(x,y) \in \mathcal{R}_i$, we can see that each strong unstable leaf $\mathcal{F}^u(x,y)$ in $\mathcal{R}_i$ is the vertical interval $\{x\}\times J_i$, and so, $F(\{x\}\times J_i)$ contains the vertical interval $\{f_i(x)\}\times J$ which contains $\{f_i(x)\}\times J_i$ for each $i=0,1,2$. Then, the orbit of the vertical interval $\{x\}\times J_i$ by $F$ contains $\{h(x)\}\times J$ with $h$ in IFS $<f_0, f_1, f_2>$.

Thus, in order to understand the orbit of $F$ we show here some dynamics properties for IFS $<f_0, f_1, f_2>$.

\begin{prop}\label{minimal-ISF}
For every $x \in \mathbb{S}^1$, the orbit $\mathcal{O}(x)$ 
by the IFS $<f_0,f_1,f_2>$ intersects the interval $I$.
\end{prop} 

\begin{proof}
Since $p_0^i \in (p_{-}^i,p_{+}^i)$ is the unique attractor for $f_i$, for $i=1,2$, it is enough to show that $\mathcal{O}(x)$ intersects $(p_{-}^1,p_{+}^2)$. Indeed, for each $h(x) \in (p_{-}^1,p_{+}^2)$, we can iterate it either by $f_1$ or $f_2$, getting either $f^k_1(h(x))$  close to the attractor $p_0^1$ or $f^k_2(h(x))$ close to $p_0^2$. 
Then, we iterate by either  $f_1$ or $f_2$,  so that the orbit enters in $I$. 
	
We prove now that for each $x$ on the circle, there is $h$ in $<f_0,f_i>$ such that $h(x)$ belongs to $(p_{-}^i,p_{+}^i)$. In order to prove that, we suppose without loss of generality that $f_0^k(x)\notin (p_{-}^i,p_{+}^i)$ for every $k\geq 0$. Then, we observe that
$$\underbrace{f_0 \circ \cdots \circ f_0}_{n-k} \circ f_i \circ \underbrace{f_0 \circ \cdots \circ f_0}_{k-1}(x)
=f_0^{n-k} (f_i (f_0^{k-1}(x)) = \mu^{n-1}\alpha_i x+f_0^k(\beta_i),$$
for $1\leq k \leq n-1$, belong to the orbit $\mathcal{O}(x)=\{h(x): h \in <f_0,f_1, f_2>\}$. Recall that $f_i(x)=\alpha_ix+\beta_i$ for $x \in \mathbb{S}^1\backslash (p_{-}^i,p_{+}^i)$ and $\{f_0^k(\beta_i):k\geq 0 \}$ is dense in $\mathbb{S}^1$. Thus, we can conclude that  $\{h(x): h \in <f_0,\, f_i>\}$ is dense in $\mathbb{S}^1$.
\end{proof}

The following holds as an immediate consequence of Proposition \ref{minimal-ISF} and  the construction of $F$.

\begin{lemma}\label{lemma:unstable}
For every $x \in \mathbb{S}^1$, the closure of $\{\{h(x)\}\times J: h \in <f_0,f_1,f_2>\}$ intersects the blender $(\mathcal{Q}, F)$. Consequently, for every $x \in \mathbb{S}^1$, it follows that 
\begin{align} 
\mathcal{Q} \cap \left( \cup_{k\geq 0} F^k(\{x\}\times J)\right) \neq \emptyset.
\end{align}
\end{lemma} 

We know, by Remark \ref{s-mfld}, that the local stable manifold $\mathcal{W}^s_{loc}(p_F)$ contains the interval $I$. Then, the set $\cup_{k\geq 0}F^k(\{x\}\times J)$ intersects $\mathcal{W}^s_{loc}(p_F)$ for every $x \in \mathbb{S}^1$. By compactness, we can choose an iterate $\cup_{j=0}^nF^j(\{x\}\times J)$ which intersects the local stable manifold $\mathcal{W}^s_{loc}(p_F)$ for each $x \in \mathbb{S}^1$. Finally, shrinking, if necessary, the neighborhood $\mathcal{U}_F$, we can see that $\cup_{j=0}^nG^j(\{x\}\times J)$ is  close enough to $\cup_{j=0}^nF^j(\{x\}\times J)$ and, using that the local stable manifold $\mathcal{W}^s_{loc}(p_G)$ depends continuously of $G$, holds $\cup_{j\geq 0} G^j(\{x\}\times J)$ intersects $\mathcal{W}^s_{loc}(p_G)$ for each $G \in \mathcal{U}_F$. Therefore, we can conclude the following.

\begin{lemma}\label{stable intersects}
There exists a neighborhood $\mathcal{U}_F$ of $F$ such that for every $G \in \mathcal{U}_F$ and every $x \in \mathbb{S}^1$, we have that
\begin{align}\label{eq:blender}
\mathcal{W}^s_{loc}(p_G,G)\cap \left(\cup_{k\geq 0} G^k(\{x\}\times J)\right) \neq \emptyset.
\end{align}
More general, for each $u$-arc $\gamma$ and every $G \in \mathcal{U}_F$, we have that the orbit $\cup_{n\geq 0} G^n(\gamma)$ intersects $\mathcal{W}^s_{loc}(p_G)$
\end{lemma}


\begin{proof}
Recall that $\mathcal{W}^s_{loc}(p_G)$ depends continuously on $G$ 
 and that $\mathcal{W}^s_{loc}(P_F)$ contains the interval $I$ (which form the blender $\mathcal{Q}=I\times J$). Since $F^n(x,y)$ converges to $p_F$ for every $(x,y) \in \mathcal{W}^s_{loc}(p_F)$, we can assume that $\{x\}\times J$ hits $\mathcal{W}^s_{loc}(p_G)$. By compactness, we can take a large positive integer $n$ such that $\cup_{j=0}^n F^j(\{x\}\times J_i)$ intersects $\mathcal{W}^s_{loc}(p_F)$ for every $x \in \mathbb{S}^1$ and $i=0,1,2$. In particular, the set $\cup_{j=0}^n F^j(\{x\}\times J_i)$ intersects $\mathcal{W}^s_{loc}(p_G)$ for every $G \in \mathcal{U}_F$. Then, shrinking $\mathcal{U}_F$ if necessary, we can assume that $G^j$ is close to $F^j$ for each $j=1,\dots, n$ so that $\cup_{j=0}^n G^j(\{x\}\times J_i)$ also intersects $\mathcal{W}^s_{loc}(p_G)$ which proves the first part of the lemma.

We know that the length of $G^n(\gamma)$ grows exponentially and keep being tangent to the unstable cone $\mathcal{C}^u_{\alpha}$, for each $u$-arc $\gamma$. Then, we can assume that $G^n(\gamma)$ crosses the stripe $\mathcal{R}_i$ for each $i=0,1,2$. Say $\gamma_0$ a component of some $G^n(\gamma)$ in $\mathcal{R}_i$. Thus, up to shrink the neighborhood $\mathcal{U}_F$, we can take by Proposition\ref{lambda-lemma} a sequence $(\gamma_k)_k$ which $\gamma_k$ is the component of $G(\gamma_{k-1})$ in $\mathcal{R}_i$ and $\gamma_k$ is ``almost'' vertical as $\{x\}\times J_i$ for large $k$ and some $x \in \mathbb{S}^1$. Then, we can take $k$ large enough such that $G^j(\gamma_k)$
 stay close to $G^j(\{x\}\times J_i)$ for several iterates, so that some iterated $G^j(\gamma_k)$ hits the local stable manifold $\mathcal{W}^s_{loc}(p_G)$.
\end{proof}

\begin{rmk}
We would like to point out that the IFS $<f_0,f_1,f_2>$ is used to build the blender for $F$ and to guarantee that every local strong unstable leaf $\mathcal{F}^u_{loc}(p)$ (``vertical'' interval) has some iterate that intersects the blender. In consequence, these properties are used to extend the same properties for all $G$ sufficiently close to $F$ and for all $u$-arc $\gamma$.
\end{rmk}

Additionally, note that as $F$ is expanding outside of the blender and 
the closure of  $\mathcal{W}^u(p_G, G)$ has nonempty interior for $G$ nearby $F$, for further details see \cite[Proposition 3.2]{He-Gan}. We can assume, up to shrinking $\mathcal{U}_F$, the following statement.

\begin{lemma}\label{u-lemma}
The image of the closure of $\mathcal{W}^u(p_G,G)$ by a large iterate of $G$ is the whole surface $\mathbb{T}^2$, for every $G \in \mathcal{U}_F$.  
\end{lemma}

\subsubsection{Creating critical points}
We make here a local perturbation to create artificially the critical points of all endomorphisms in $\mathcal{U}_F$. Artificially means that the critical points will be created in such a way that all the previous properties obtained so far are preserved.


Let $B_{t,s}$ be a box centered at the origin with sides of length $t$ and $s$ respectively, with $t,\, s>0$ small enough, and two $C^{\infty}$ maps $\psi :\R\to [0,\mu+\frac{1}{2}]$ and $\varphi:\mathbb{R}\to \mathbb{R}$ such as in Figure~\ref{psi-varphi} verifying:
\begin{itemize}
	\item $\psi(x)=\psi(-x)$ and $\mu< \psi (0)<\mu+\frac{1}{2}$;
	\item $\varphi(0)=0, \,\min\{\varphi'\}\geq-\frac{\lambda-\mu}{\mu+1}$, and $\max\{ \varphi'\}=1$.
\end{itemize}

\begin{figure}[!h]
\centering
\includegraphics[scale=0.5]{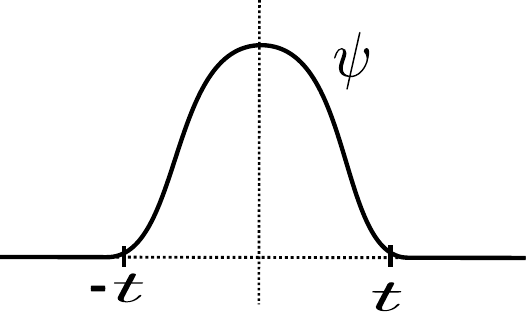}\hspace{1cm}
\includegraphics[scale=0.6]{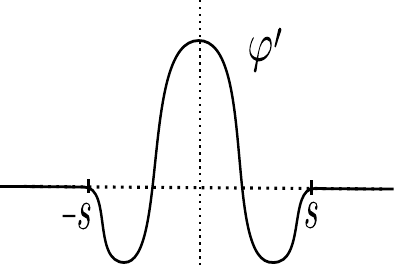}
\caption{The graphs of $\psi$ and $\varphi'$, respectively\protect\footnote.}
\label{psi-varphi}
\end{figure}
\footnotetext{The graph and conditions of $\varphi'$ are slightly different from the one in \cite{ILP}, since the one used in \cite{ILP} does not satisfy the properties needed to preserve the weak hyperbolic structure.}
Given $(x_0,y_0) \in \mathbb{T}^2$, we define $F_{t,s}:\mathbb{T}^2\to \mathbb{T}^{2}$ by 
$$F_{t,s}(x,y)=(f_y(x) - \Phi (x,y),\lambda y ({\rm{mod}}\,\, 1)),$$
where $\Phi(x,y)=\varphi(x-x_0)\psi(y-y_0)$ for every $|x-x_0|<s$, $|y-y_0|<t$; and $\Phi(x,y)\!=\!0$, otherwise. 
Adapting the arguments used in \cite[Lemma 2.1.1]{ILP} follows that, for $t,s$ small enough, $F_{t,s}$ still preserves the weak hyperbolicity structure, that is the unstable cone family is preserved.


Choose $(x_0,y_0) \in~\mathbb{T}^2$ such that $B_{t,s}$ centered at $(x_0,y_0)$ is away from the stripe $\mathcal{R}_J$, this ball is disjoint from the blender region and, in consequences, the previous lemmas hold for $F_{t,s}$. Note that the \textit{critical points} of $F_{t,s}$ are given by 
\begin{align}
\mathrm{Cr}(F_{t,s})=\{(x,y)\in \mathbb{T}^2: \mu-\psi(y-y_0)\varphi'(x-x_0)=0\}.
\end{align}
Since $F$ is a $C^1$-endomorphism having points on which the determinant of $DF$ is negative (e.g., in $(x_0,y_0)$) and positive, we have persistence of the critical points in the $C^1$-topology.

\begin{rmk}\label{rmk0}
Note that  $|\varphi|$  goes to zero as $s$ goes to zero. Hence, it holds that $F_{t ,s}$ goes to $F$ in the $C^{0}$ topology, when  $t$ and $s$ go to zero. In particular, since $F$ was constructed $C^0$-close to $L$, it follows that $F_{t ,s}$ is $C^0$-close to $L$. Hence,  $F_{t,s}$ is homotopic to $L$.
\end{rmk}

From now on, by slight abuse of notation, we fix $t,s$ small enough, and call $F=F_{t,s}$. Finally, we prove that $F$ is robustly transitive, and therefore, satisfies 
Theorem~\ref{thm A}.

\subsubsection{Proof of Theorem \ref{thm A} for the periodic case}

Let $U$ and $V$  two open sets in $\mathbb{T}^2$ and $G\in \mathcal{U}_F$. By Lemma~\ref{stable intersects} and Lemma~\ref{u-lemma}, there are an iterate of an $u$-arc $\gamma$ contained in $U$ which intersects $\mathcal{W}^s_{loc}(p_G,G)$ and some iterate of the closure of $\mathcal{W}^u_{loc}(p_G,G)$ which intersects $V$. Next, we take a ``horizontal'' arc $\beta$ contained in some preimage of $V$ that hits $\mathcal{W}^u_{loc}(p_G, G)$ and is transversal to the unstable cone. Finally, since $p_G$ is a hyperbolic fixed point, we conclude that $G^n(\gamma)$ has a component $\gamma_n$ which is ``almost vertical'' and is getting close to $\mathcal{W}^u_{loc}(p_G,G)$, hitting $\beta$. Therefore, $G$ is transitive and this concludes the proof for the periodic case.    


\subsection{Irrational Case}\label{IC}

From now on, we assume that the eigenvalues $\mu$ and $\lambda$ are irrational with $1 < |\mu| \ll |\lambda|$, recalling that $1<|\mu|\ll |\lambda|$ means that $|\lambda|$ is larger enough than $|\mu|$
such that the correspondent eigenspaces generate a dominated splitting. Let $v^c$ and $v^u$ be the unit eigenvector of $L$ associated to the eigenvalues $\mu$ and $\lambda$, respectively. Recall  that the eigenspaces $E^c$ and $E^u$ are generated by $v^c$ and $v^u$ and the leaf $\mathcal{F}^{\sigma}(q)$ is the line $\{q+tv^{\sigma}:t \in \mathbb{R}\}$ on $\mathbb{T}^2$, for every $q \in \mathbb{T}^2$ and $\sigma=c,u$. Given $q \in \mathbb{T}^2$, denote by $\mathcal{F}^{\sigma,\pm}(q)$ the connected components of $\mathcal{F}^{\sigma}(q)\backslash \{q\}$, that is, $\mathcal{F}^{\sigma,+}(q)$ is the semi-line $\{q+tv^{\sigma}:t\geq 0\}$, similarly for $\mathcal{F}^{\sigma,-}(q)$ with $t\leq 0$. Analogously, we define $\mathcal{F}^{\sigma,\pm}_{loc}(q)$ as the components of $\mathcal{F}^{\sigma}_{loc}(q)\backslash \{q\}$.

\subsubsection{Sketch of the proof}
Since the proof is slightly more technical than the periodic case, we present here the strategy and main steps of the proof. However, the general idea is the same as for the periodic case.

\smallskip

\noindent
\textit{Step I:} First, we create a blender. Using the density of $\mathcal{F}^{u,+}(p),$  with $p=(0,0),$ on the torus, we make a local deformation of $L$ far away from $p$ in order to create a \textit{cycle} at $p$, let us call by $F$ the perturbed map;
that is, 
we find an interval $I$ in $\mathcal{F}^{u,+}_{loc}(p)$ and a positive integer $n$ such that $F^n(I)$ contains $\mathcal{F}^{u,+}_{loc}(p)$. 

\begin{rmk}
It should be noted that for the periodic case, the strong unstable manifold $\mathcal{F}^u(x,y)$ is the closed curve $\{x\}\times \mathbb{S}^1$. Hence, for the periodic case it is not needed to make a deformation to create a cycle, since any interval $I$ in $\mathcal{F}^u(x,y)$ has an iterate by $L$ which is the whole $\mathcal{F}^u(x,y)$.
\end{rmk}

Then, we perturbe $F$ in a small neighborhood of $p$ where $F$ coincide with $L$, so that the expanding fixed point $p$ of $L$ becomes a hyperbolic saddle point. This new map is called ``\textit{Derived from Anosov}'' of $F$, DA for short. This perturbation is done in such a way that the DA-endomorphism, which by abuse of notation we also denote by $F$, preserves the cycle created previously. In addition, it also preserves the unstable cone field $\mathcal{C}_{\alpha}^u$ defined in \eqref{coneL}, and $\min \|DF(v)\|\geq \lambda'$ for some $\lambda'>1$ and for every unit vector $v$ in $\mathcal{C}_{\alpha}^u$.

Finally, we get a box $\mathcal{Q}=I \times J$ where $I$ is an interval in $\mathcal{W}^s_{loc}(p,F)$ and $J$ is an interval in $\mathcal{W}^u_{loc}(p,F)$ where $p$ is a common extreme of $I$ and $J$. Next, we choose an interval $J_1 \subseteq J$ containing $p$ which $F(J_1) \supseteq J$; and, using the existence of a cycle for $F$, choose another interval $J_2 \subseteq J$ and $n \in \mathbb{Z}^{+}$ such that $F^n(J_2) \supseteq J$. Furthermore, we show that there is $\ell \in \mathbb{Z}^{+}$ such that $F^{\ell}(x,y)=(f_i(x),\lambda^{\ell} y \, (\mathrm{mod}\,\,1))$ on $(x,y) \in \mathcal{R}_i=I\times J_i$, where $i=1,2,$ and $f_1,f_2:I \to I$ are as Figure~\ref{IFS-blender-2}. Then, we conclude that $(\mathcal{Q}, F)$ is a blender and, therefore, a blender for every $G$ close enough to $F$ in the $C^1$-topology. 

\medskip

\noindent
\textit{Step II:} After ensured the existence of a blender, we prove that every $u$-arc has some iterate by $F$ that intersect the blender. Then, we extend this property for every endomorphism $C^1$-close to $F$, using the density of all the leaves of $\mathcal{F}^u$	and the fact that $F$ coincides with $L$ out of the region of perturbation. The rest of the proof follows from adapting the arguments of the periodic case to conclude  Theorem~\ref{thm A} for the irrational case.

\begin{rmk}
Recall that in the periodic case the IFS $<f_0,f_1,f_2>$ is used to guarantee that every $u$-arc has some iterate which intersect the blender.
\end{rmk}

\subsubsection{Step I: creating a blender}
We now prove the first step given in the previous sketch, recalling that $p=(0,0)$ is an expanding fixed point of $L$. 

\begin{lemma}[Cycle]\label{B-irracional}
For every neighborhood $\mathcal{U}$ of the identity in ${\rm{End}}^1(\mathbb{T}^2)$, there exists $h_t \in \mathcal{U}$ such that
\begin{enumerate}[label=$\mathrm{(\alph*)}$]
\item $F_t:=h_t\circ L$ is $C^1$-close to $L$;
\item there exists $I \subset \mathcal{W}^u_{loc}(p,F_t)\backslash\{p\}$ an interval such that $F_t^n(I)$ contains the local unstable manifold $\mathcal{W}^u_{loc}(p, F_t)$. 
\end{enumerate}
\end{lemma}

\begin{proof}
Observe that if there are a positive integer $n$ and an interval $I$ contained in $\mathcal{F}^u_{loc}(p)\backslash\{p\}$ such that $L^n(I)$ contains the local unstable manifold $\mathcal{F}^u_{loc}(p)$ of $L$, then there is nothing to prove and $L$ verifies the lemma. Hence, let us assume that $L^n(\mathcal{F}^u_{loc}(p)\backslash\{p\})$ does not contain $p$ for every positive integer $n$.

Let us prove item (a). Note that given $q$ on the torus, we may rewrite $L$ locally as $L(x,y)=q+(\mu x, \lambda y)$ for every $(x,y)$ on the rectangle $I^c_{r_1}\times I^u_{r_2}$, 
where $I^{\sigma}_r$ is an interval on $E^{\sigma}$ centered at the origin and radius $r$, for $\sigma=c,u$.  Let $\varphi:\mathbb{R} \to \mathbb{R}$ be a bump function such that $\varphi(s)=1$, if $|s|\leq 3/4$, and $\varphi(s)=0$, if $|s|\geq 1$.  The projection of the rectangle  $I^c_{r_1} \times I^u_{r_2}$ centered at the origin defined on the tangent bundle, by slightly of abuse, is denoted by  $I^c_{r_1} \times I^u_{r_2},$ observing that now it is a rectangle centered at $q$ on the torus.
Then, define $h_t:\mathbb{T}^2 \to \mathbb{T}^2$ as $h_t = Id$ outside of the rectangle $I^c_{r_1} \times I^u_{r_2}$ centered at $q$; and define $h_t$ on $I^c_{r_1} \times I^u_{r_2}$ as the following map, 
\begin{align}\label{Pert-Id}
h_t(q+ (x,y))=q+(x +t\Phi(x,y), y),\, \forall (x,y)\in   I^c_{r_1} \times I^u_{r_2},
\end{align}
where $\Phi(x,y)=\varphi(x/{r_1})\varphi(y/r_2)$. For simplicity, we will omit the point $q$ in \eqref{Pert-Id}. Note that $h_t$ deforms the strong unstable leaf that crosses the rectangle on the $v^c$-direction, and has the form $h_t(x,y)=(x+t\varphi(x/r_1),y)$ for every $(x,y) \in I_{r_1}^c\times I^u_{2r_2/3}$. Denote by $\mathcal{R}$ the rectangle $I_{r_1}^c\times I^u_{2r_2/3}$ 
 and notice that $h_t$ translates interval in $\mathcal{R}$ parallel to the strong unstable direction. Fix $r_1$ and $r_2$ small enough such that for $0<t<r_1$, $h_t$ is $C^1$-close to the identity and $F_t=h_t \circ L$ is $C^1$-close to $L$. It concludes the proof of item~(a). 



The local translation $h_t$ is supported in the region $\mathcal{R}$ that will be stablished in Claim II below. 

\medskip

\noindent
\textit{Claim I:} There exists a sequence $(I^u_j)_j$ of strong unstable intervals contained in $\mathcal{F}^{u,+}(p)\cap \mathcal{R}$ with the following properties:
\begin{enumerate}[label=$\mathrm{(\roman*)}$]
\item there exists $(n_j)_j$ positive integers so that $I_j^u \subseteq L^{n_{j}}(I^u_{j-1})$ for each $j \geq 1$; 
\item for each $1\leq  k <n_j$, the interval $L^k(I^u_{j-1})$ does not intersect  $\mathcal{R}$.
\end{enumerate}

\smallskip

Recall that any internal $I$ contained in some strong unstable leaf of $\mathcal{F}^u$ verifies that the length of $L^n(I)$ grows exponentially and got dense on the whole torus. Thus, we can take a positive integer $n_0$ as the first time that $L^{n}(\mathcal{F}^{u,+}_{loc}(p))$ crosses the rectangle $\mathcal{R}$. Denote the component of $L^{n_0}(\mathcal{F}^{u,+}_{loc}(p))$ that crosses the rectangle $\mathcal{R}$ by $I_0^u$. We now observe that $L^n(I^u_0)$ also get dense as $n$ increase. Again, we can choose $n_1$ as the first time that $L^n(I^u_0)$ crosses $\mathcal{R}$; and denote by $I_1^u$ such component. It is possible because $L^n(I_0^u)$ is getting dense on the torus by parallel segments and so if $L^n(I_0^u)$ hits on $\mathcal{R}$, but does not cross it, then we remove the intersection $L^n(I_0^u)\cap \mathcal{R}$ and keep iterating $L^n(I_0^u)\backslash (L^n(I_0^u)\cap \mathcal{R})$ until it crosses the rectangle $\mathcal{R}$ for the first time. Repeating that process indefinitely, we obtain a sequence $(I^u_j)_j$ in $\mathcal{R}$ such that $I^u_j$'s are contained in $\mathcal{F}^{u,+}(p)$ and verify $I^u_j$ contained in $L^{n_j}(I_{j-1}^u)$ for every $j\geq 1$. Furthermore, by the choices of $n_j$'s, we have that $L^k(I^u_j)$ does not hit the region $\mathcal{R}$ for each $1\leq k < n_j$. This concludes the claim.

\smallskip

\noindent
\textit{Claim II:} Given $\varepsilon >0$, there is $n\geq 1$ and an interval $I \subseteq \mathcal{F}^{u,+}_{loc}(p)\backslash\{p\}$ so that $F_t^n(I) \subseteq \mathcal{F}^u_{loc}(p)$ containing $p$, for some $0< t \leq \varepsilon$.

\smallskip

Since $F_t$ coincides with $L$ outside of the rectangle $I^c_{r_1}\times I^u_{r_2}$, we get that \linebreak
$\mathcal{W}^u_{loc}(p,F_t)=\mathcal{F}^u_{loc}(p)$. Then, Claim II allows us to conclude the proof of item (b), and so, complete the proof of Lemma~\ref{B-irracional}.

Let us now fix the center of the rectangle
$I^c_{r_1}\times I^u_{r_2},$ that is fix $q$ such that
$q_1=q+(r_1,0)$ is the center of the vertical right boundary of the rectangle and it belongs to $L^{-1}(p)$. By Claim I, consider a sequence of positive integers $(n_j)_j$ and intervals $(I_j^u)_j$ contained in $\mathcal{R}$ such that $I_0^u$ is contained in $L^{n_0}(\mathcal{F}^{u,+}_{loc}(p))$ and $I^u_j \subseteq L^{n_j}(I^u_{j-1})$ and $L^k(I^u_{j-1})$ does not hit on $I^c_{r_1}\times I^u_{r_2}$ for each $1\leq k < n_j$. Assume without loss of generality that $I^u_j=L^{n_j}(I^u_0)$ for some $n_j$ large enough so that $\mu^{n_j}\varepsilon \geq r_1$.

We are now able to prove Claim II. Since, for each $0< t\leq \varepsilon$, the map $F_t$ is a $v^c$-translation on $I^c_{r_1}\times I^u_{r_2}$ and $F_t=L$ outside of the rectangle $I^c_{r_1}\times I^u_{r_2}$, we have that the image of each leaf $\mathcal{F}^u(x,y)$ by $F_t$ is the leaf $\mathcal{F}^c(L(x,y))$ for every $t$. Moreover, we also have that there is $I\subseteq \mathcal{F}^{u,+}(p)$ such that its $n_0$-iterate by $F_t$ is equal to $h_t\circ L^{n_0}(I),$ that is a $v^c$-translation of $I^u_0$ on $\mathcal{R}$. We denote it by $I_t$ and take the family $\mathcal{R}_0=\{I_t:0\leq t\leq \varepsilon\}$. Define $\mathcal{R}_j$ as the set $\{F^{n_j}_t(I_t): 0\leq t \leq \varepsilon\}$. Then, using that $F_t$ preserves the weak unstable foliation $\mathcal{F}^c$ for every $t$, we have  that $\mathcal{R}_j$ is a family of intervals $(I_t^j)$ for $0\leq t \leq \mu^{n_j}\varepsilon$ contained in $\mathcal{R}$ such that each interval $I^j_t$ is equal to $F^{n_j}_t(I_{t'})$ for some $0\leq t' \leq \varepsilon$. In particular, $I_0^j$ is equal to $I^u_j=L^{n_j}(I_0^u)$.


Therefore, we conclude that there is $I \subseteq \mathcal{F}^{u,+}_{loc}(p)$ and $0<t\leq \varepsilon$ such that $F^{n_0+n_j}_t(I)=F^{n_j}_t(I)$ contains the interval $\{r_1\}\times I^u_{2r_2 /3}$, which implies that $F_t^{n}(I)$ is an interval in $\mathcal{F}^u_{loc}(p)$ containing $p$. This complete Claim II and, consequently, proves that $F_t$ has a cycle.
\end{proof}

Now we get a Derived from Anosov (DA) from $F_t$ in order to create a blender.

\begin{lemma}[Derived from Anosov]
Let $F_t$ be the $C^1$-map given by Lemma~\ref{B-irracional}. There is a DA from $F_t$, denoted by $F$, which has a blender for some iterated $F^n$.
\end{lemma}

\begin{proof}
Consider a small ball $B$ centered at the fixed point $p=(0,0)$ and radius $r>0$. Fix $0<\eta < r$, define $\xi: \R^2 \to [0,1]$ a bump function such that $\xi(z)=0$ in $B^c$ and $\xi(z)=1$ for $ \|z\|\leq \eta$. Define the family of diffeomorphisms $\Theta_s: \mathbb{T}^2 \to \mathbb{T}^2$ by $\Theta_s(x,y)=(x e^{-s \xi((x,y))},y)$ in the $(v^c,v^u)$-coordinates, if $(x,y) \in B$, and $\Theta_s$ is equal to the identity on $\mathbb{T}^2\backslash B$.
Then, the endomorphism $F_{s,t}$ given by $\Theta_{s}\circ F_t$ satisfies the following properties:
\begin{itemize}
\item $p$ is a hyperbolic fixed point and the derivative of $DF_{s,t}$ at $p$ in the $(v^c,v^u)$-coordinates is a diagonal matrix $DF_{s,t}(p)=\diag(e^{-s}\mu,\lambda)$, denote $e^{-s}\mu$ by $\mu_{s}$;
\item the map $F_{s,t}=F_t$ in $\mathbb{T}^2\backslash B$ which is an expanding map. Furthermore, as $F_t=L$ outside the box $\mathcal{R}$, one has that $F_{s,t}=L$ outside $B\cup \mathcal{R}$;
\item for every $(x,y) \in B$ with $\|(x,y)\|<\eta$, we have that $F_{s,t}(x,y)=(\mu_{s}x, \lambda y)$;
\item we assume that $\mathcal{W}^u_N(p,F_{s,t})=\mathcal{W}^u_N(p,F_{t})$ for every $s \in \mathbb{R}$, where $N>0$ is chosen so that it has a cycle.
\end{itemize}
	
\begin{rmk}\label{iterado-m}
The last item above is possible since $F_t$ admits an interval $I \subseteq \mathcal{W}^u_{loc}(p,F_t)$ such that $F_t^m(I) \supseteq \mathcal{W}^u_{loc}(p,F_t)$ for some $m$, $\mathcal{W}^u_{loc}(p,F_{s,t})$ coincide with the local strong unstable $\mathcal{W}^u_{loc}(p,F_t)$ which is $\mathcal{F}^u_{loc}(p)$  and $B$ can be chosen small enough such that $F^k_t(I)$ does not hit on $B$ for $1\leq k \leq m-1$.
\end{rmk}
	
Let $\mathcal{Q}=[0,\eta']\times [0,\eta']$ be the box in $(v^c,v^u)$-coordinates with $0<\eta'<\eta$ contained in $B$, where $B$ was stablished above. Note that $[0,\eta]\times \{0\}$ is contained in $\mathcal{W}^s_{loc}(p,F_{s,t})$ and $\{0\}\times [0,\eta']$ is contained in $\mathcal{W}^u_{loc}(p,F_{s,t})$. Then, take two disjoint intervals $J_1$ and $J_2$ as folllows,  $J_1 \subseteq \{0\} \times [0,\eta']$ with $p \in \partial J_1$ and $J_2 \subseteq \{0\}\times [0,\eta']$ on which $F^m_{s,t}(J_2) \supseteq \{0\}\times [0,\eta']$ (it is possible that exist a cycle). Moreover, we assume that $F_{s,t}(J_1)$ contains $\{0\}\times [0,\eta']$. In particular, it contains the interval $J_2$.
Let $\mathcal{R}_1$ and $\mathcal{R}_2$ be two stripe in $\mathcal{Q}$ of the form $\mathcal{R}_1=[0,\eta']\times J_{1}$ and $\mathcal{R}_2=[0,\eta']\times J_{2}$, see Figure~\ref{IFS-irrational}; and if necessary we take $\lambda$ large enough to guarantee that $F_{s,t}^k(\mathcal{R}_2)$ does hit on $B$ until the $m$-iterate.
	
For simplicity, denote by $\mathcal{F}^u(x)$ the strong unstable leaf containing $x$ in $\mathcal{Q}$ and by $\mathcal{F}^u_{i}(x)$ the connected component in $\mathcal{R}_i$ for $i=1,2$. Note that, by construction, $F_{s,t}^{m+k}$ preserves the unstable foliation $\mathcal{F}_i^u$ and if $t'>t$ then $F_{s,t'}^{m}(J_2)$ is of the form $\mathcal{F}^u(x)$ in $\mathcal{Q}$ with $0<x<\eta'$, recalling that $m$ was stablished in Remark~\ref{iterado-m}. Hence, $F_{s,t}^{m+k}$ induces a one-dimensional map $f_i:J_i/_\sim \to J_i/_\sim$, where $x\sim y$ if and only if $y\in \mathcal{F}^u(x)$ (it is easy to see that $J_i/_\sim=[0,\eta']$), defining $f_i(x)$ as the leaf $F_{s,t}^{m+k}(\mathcal{F}^u_i(x))$. Moreover, we have that $|f_1'|=|\mu_s|^{m+k}$ and $|f_2'|=|\lambda|^{m}|\mu_s|^k$. Then, we choose $k\geq 1$, $t>0$ and $s>0$ such that $|f'_1|,|f_2'|<1$, and $f_1, f_2$ are as in Figure~\ref{IFS-blender-2}. Denote such $F_{s,t}$ simply by $F$ which has a blender for $F^{m+k}$.
\end{proof}

\begin{figure}[!h]
\centering
\includegraphics[scale=0.8]{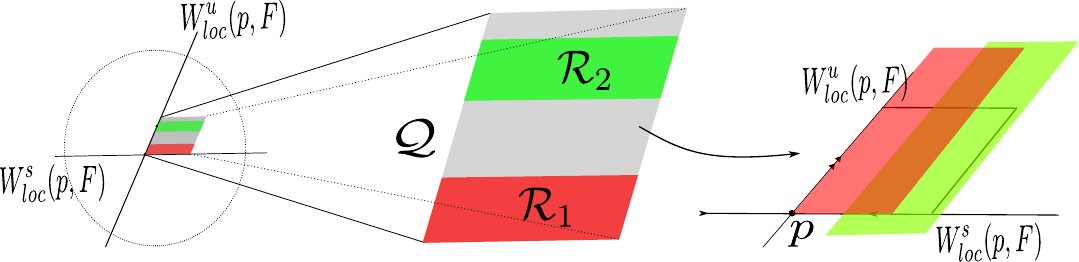}
\caption{The blender for an iterated of $F$.} \label{IFS-irrational}
\end{figure}

From now on, denote by $\mathcal{U}_F$ the neighborhood of $F$ where $(\mathcal{Q}, G)$ is a blender for every $G\in \mathcal{U}_F$. Observe that the endomorphism $F:\mathbb{T}^2 \to \mathbb{T}^2$ is a DA-endomorphism from a $C^0$-perturbation of $L$ such that  it coincides with the linear endomorphism $L$ away from the support. It is easy to check that by construction one can assume that every endomorphism $G \in \mathcal{U}_F$ has the hyperbolic structure preserving the cone field $\mathcal{C}^u_{\alpha}$ and verifying $\min\|DG(v)\|\geq \lambda'>1$.

\subsubsection{Step II: crossing the blender}
Here, we guarantee that every $u$-arc $\gamma$ admits an iterate by $G^n$ which cross the blender. More precisely,
\begin{lemma}\label{lemma:crossingblender}
There is a neighborhood $\mathcal{U}_F$ of $F$ such that for every $G \in \mathcal{U}_F$ and every $u$-arc $\gamma$, there exists a component $\hat{\gamma}$ of $G^n(\gamma)$ which cross the blender $\mathcal{Q}$ and intersect $\mathcal{W}^s_{loc}(p_g,G)$.
\end{lemma}

The idea of the proof follows from the following steps. First, we observe that the strong unstable foliation $\mathcal{F}^u$ is minimal, that is, every leaf of  $\mathcal{F}^u$ is dense. Then, fix a large number $K$ such that every segment on $\mathcal{F}^u$ whose length is greater than $K$ cross the blender. Moreover, we can fix $\alpha>0$ small enough such that any arc $\gamma$ of length greater than $K$ tangent to the cone-field $\mathcal{C}_{\alpha}$ also crosses the blender. Thus, now we may take a neighborhood $\mathcal{U}_F$ of $F$ such that every endomorphism $G$ in it also preserves the cone-field $\mathcal{C}_{\alpha}$ and every $u$-arc $\gamma$ has some iterated by $G$ which crosses the blender. Since $F$ restricted to $\mathcal{R}_i$ has the form $F(x,y)=(f_i(x),\lambda y)$, we can adapt Proposition~\ref{lambda-lemma} for this setting. Finally, we repeat the same arguments as in the proof of Lemma~\ref{stable intersects} to conclude the lemma.

\smallskip

Now we are able to prove the robustness of transitivity for $F$. 
\subsubsection{Proof of Theorem \ref{thm A} for the irrational case} 

The proof follows from the same arguments as in the periodic case. We put artificially the critical points as in the previous case in such way those lemmas in Steps I and II keep holding. Thus, one can conclude the proof of Theorem \ref{thm A} repeating the same arguments done for the periodic case. \qed


\section*{Acknowledgments}
The authors are grateful to E. Pujals and R. Potrie for insightful comments to improve this work. 
The first author would like to thank to UFBA and the second author to UFAL and ICTP for the nice enviroment and support during the preparation of this work. 



%

\end{document}